\documentclass[10pt]{amsart}  
\usepackage{amsmath,amsthm,amssymb,amsfonts,graphicx, hyperref}

\title{Spectral bounds for the independence ratio and the chromatic number of an operator}

\author{Christine Bachoc}
\address{C.~Bachoc, Universit\'e Bordeaux I, Institut de
  Math\'ematiques de Bordeaux, 351,
cours de la Li\-b\'e\-ration, 33405 Talence, France}
\email{bachoc@math.u-bordeaux1.fr}

\author{Evan DeCorte} 
\address{P.E.B~DeCorte, Delft Institute of Applied Mathematics, Technical University of Delft, P.O. Box 5031, 2600 GA Delft, The Netherlands}
\email{p.e.b.decorte@tudelft.nl}

\author{Fernando M\'ario de Oliveira Filho}
\address{F.M.~de Oliveira Filho, Institut f\"ur Mathematik,
Freie Universit\"at Berlin, Arnim\-allee 2, 14195 Berlin, Germany}
\email{fmario@mi.fu-berlin.de}

\author{Frank Vallentin} 
\address{F.~Vallentin, Mathematisches Institut, Universit\"at zu
  K\"oln, Weyertal 86--90, 50931 K\"oln, Germany}
\email{frank.vallentin@uni-koeln.de}

\thanks{The second and fourth author were supported by Vidi grant
  639.032.917 from the Netherlands Organization for Scientific
  Research (NWO). The third author was supported by Rubicon grant
  680-50-1014 from the Netherlands Organization for Scientific
  Research (NWO)}

\subjclass{47B25, 05C50}

\keywords{independence number/ratio, measurable chromatic number, infinite graph,
  adjacency operator, Lov\'asz $\vartheta$-number, Hoffman bound}

\date{June 19, 2013}

\newcommand{\R}{\mathbb{R}}
\newcommand{\N}{\mathbb{N}}

\newcommand{\C}{\mathbb{C}}

\newcommand{\chim}{\chi_{\rm m}}
\newcommand{\Tcal}{\mathcal{T}}
\newcommand{\Rcal}{\mathcal{R}}

\newtheorem{defin}{Definition}[section]
\newtheorem{definition}[defin]{Definition}

\newtheorem{theorem}[defin]{Theorem}

\DeclareMathOperator{\vol}{vol}

\DeclareMathOperator{\trace}{trace}

\DeclareMathOperator{\Real}{Re}

\begin{document}

\begin{abstract}
  We define the independence ratio and the chromatic number for
  bounded, self-adjoint operators on an $L^2$-space by extending the
  definitions for the adjacency matrix of finite graphs. In analogy to
  the Hoffman bounds for finite graphs, we give bounds for these
  parameters in terms of the numerical range of the operator. This
  provides a theoretical framework in which many packing and coloring
  problems for finite and infinite graphs can be conveniently studied
  with the help of harmonic analysis and convex optimization. The
  theory is applied to infinite geometric graphs on Euclidean space and on the
  unit sphere.
\end{abstract} 

\maketitle

\markboth{C.~Bachoc, P.E.B.~DeCorte, F.M.~de Oliveira Filho, F.~Vallentin}{Spectral bounds for $\overline{\alpha}$ and $\chi$ of an operator}

\section{Introduction}

The independence number~$\alpha$ and the chromatic number~$\chi$ are
invariants of finite graphs which are computationally difficult to
determine in general. A by now classical result due to Hoffman~\cite[(1.6), (4.2)]{Hoffman}
gives a relatively easy way to provide upper and lower bounds in terms
of the graph's spectrum. Hoffman gave a bound for the chromatic number of a
graph~$G$ having at least one edge in terms of the smallest eigenvalue~$m(A)$ and
largest eigenvalue~$M(A)$ of the adjacency matrix~$A$ of~$G$:
\begin{equation}
\label{eq:chibound}
\chi(G) \geq \frac{M(A) - m(A)}{-m(A)}.
\end{equation}
From the proof of this inequality it is clear that the argument also
works if one replaces the adjacency matrix with any symmetric matrix
whose support is contained in the support of~$A$. One can therefore
maximize the bound by adjusting the entries in the support. This
maximum exists and is the Lov\'asz~$\vartheta$-number of the
complement of~$G$~\cite{Lovasz}. As an application of the
$\vartheta$-number, Lov\'asz gave the following spectral bound for the
independence ratio~$\overline{\alpha}$ of a regular graph on $n$
vertices~\cite [Theorem 9]{Lovasz}:
\begin{equation}
\label{eq:alphabound}
\overline{\alpha}(G) = \frac{\alpha(G)}{n} \leq \frac{-m(A)}{M(A) - m(A)},
\end{equation}
where $A$ is the adjacency matrix of $G$.

Since its first appearance in 1979 the $\vartheta$-number became a
fundamental tool in combinatorial optimization. Lov\'asz~\cite{Lovasz}
gave a proof (from the book, see Aigner and
Ziegler~\cite{AignerZiegler}) that the Shannon capacity of the
pentagon equals~$\sqrt{5}$.  McEliece, Rodemich, and Rumsey
\cite{McElieceRodemichRumsey} and independently Schrijver
\cite{Schrijver1979a} showed that a strengthening of the
$\vartheta$-number gives Delsarte's linear programming bound in coding
theory~\cite{Delsarte}. Gr\"otschel, Lov\'asz, and
Schrijver~\cite{GroetschelLovaszSchrijver} used the $\vartheta$-number
and the related $\vartheta$-body to characterize perfect graphs and to
give polynomial time algorithms to solve the independent set problem
and the coloring problem for perfect graphs. Recently, Bachoc,
P\^echer, and Thi\'ery \cite{BachocPecherThiery} extended these
results to the circular chromatic number and circular perfect
graphs. Gouveia, Parrilo, and Thomas \cite{GouveiaParriloThomas}
extended the theory of $\vartheta$-bodies from graphs to polynomial
ideals. Computing the $\vartheta$-number of a graph is a basic
subroutine in the design of many approximation algorithms, see e.g.\
Karger, Motwani, and Sudan~\cite{KargerMotwaniSudan}, Kleinberg,
Goemans~\cite{KleinbergGoemans}, Alon, Makarychev, Makarychev, and
Naor~\cite{AlonMakarychevMakarychevNaor}, and Bri\"et, Oliveira, and
Vallentin~\cite{BrietOliveiraVallentin}. Bachoc, Nebe, Oliveira, and
Vallentin~\cite{BachocNebeOliveiraVallentin} generalized the
$\vartheta$-number from finite graphs to infinite graphs whose vertex
sets are compact metric spaces. Recent striking applications in
extremal combinatorics are the proof of an Erd\H{o}s-Ko-Rado-type
theorem for permutations by Ellis, Friedgut, and
Pilpel~\cite{EllisFriedgutPilpel} and the proof of the
Simonovits-S\'os conjecture of triangle intersecting families of
graphs by Ellis, Friedgut, and Filmus~\cite{EllisFriedgutFilmus}.

By looking at the adjacency matrix of a finite graph $G = (V, E)$ as
an operator on $L^2(V)$, we extend the definitions of independence
number and chromatic number from finite graphs to bounded,
self-adjoint operators on $L^2(V)$ where $V$ is now a measure
space. In this paper we develop a theory for bounding these parameters in terms
of the spectrum of the operator, thereby extending Hoffman's and
Lov\'asz' results from finite to infinite graphs.

The main body of the theory is presented in Section~\ref{sec:spectral
  bounds}. There are several recent results which can be conveniently
interpreted as examples of our theory: In Section~\ref{sec:euclidean}
we compute spectral bounds for graphs defined on the Euclidean space
which are invariant under translations. Thereby we can recover results
from Oliveira and Vallentin~\cite{OliveiraVallentin},
Kolountzakis~\cite{Kolountzakis}, and Steinhardt~\cite{Steinhardt}. In
Section~\ref{sec:unit sphere} we determine spectral bounds for distance
graphs defined on the unit sphere and generalize results from Bachoc,
Nebe, Oliveira, and Vallentin~\cite{BachocNebeOliveiraVallentin}.

\section{Spectral bounds for bounded, self-adjoint operators}
\label{sec:spectral bounds}

\subsection{Some graph theory}
\label{sec:graph-theory}

Let $G = (V,E)$ be a finite, undirected graph on~$n$ vertices, with
vertex set $V$ and edge set $E \subseteq \binom{V}{2}$. The \emph{adjacency
matrix} of $G$ is the symmetric matrix $A \in \R^{V \times V}$ defined by
\begin{equation*}
A(v,w) = 
\left\{
\begin{array}{ll}
1 & \text{if $\{v,w\} \in E$,}\\
0 & \text{otherwise.}
\end{array}
\right.
\end{equation*}
The eigenvalues of~$A$ are real; we order them decreasingly
\begin{equation*}
\lambda_1(A) \geq \lambda_2(A) \geq \cdots \geq \lambda_{n}(A).
\end{equation*}
The smallest and largest eigenvalues of~$A$ are most important for us,
so we set $m(A) = \lambda_n(A)$, and $M(A) = \lambda_1(A)$.

An \emph{independent set} of $G$ is a subset of the vertex set in
which no two vertices are adjacent. The \emph{independence
  number}~$\alpha(G)$ of~$G$ is the cardinality of a largest
independent set. The \emph{chromatic number}~$\chi(G)$ of~$G$ is the smallest
number~$k$ so that one can partition the vertex set $V$ into $k$
independent sets. These independent sets are often called \emph{color
  classes}. 

An observation that will be key to our generalization in the next
section is the following: A set $I \subseteq V$ is independent if and
only if
\begin{equation}
\label{characterization}
\sum_{v \in V} \sum_{w \in V} A(v,w) f(v) f(w) = 0
\end{equation}
holds for all~$f \in \R^V$ supported on~$I$, i.e., for all $f$ such
that~$f(v) = 0$ whenever~$v \not\in I$. 

Hoffman~\cite{Hoffman} gave the following spectral bound for~$\chi$:
\begin{equation*}
\chi(G) \geq \frac{M(A) - m(A)}{-m(A)}.
\end{equation*}

A possibly stronger bound is given by the Lov\'asz $\vartheta$-number
of the complement of the graph. One of many possible definitions (see
Knuth~\cite{Knuth} for a survey) of the~$\vartheta$-number of a
graph~$G$, denoted by~$\vartheta(G)$, is as the optimal value of the
following semidefinite program:
\begin{equation}
\label{eq:sdp}
\begin{array}{rl}
\max &\sum_{v,w \in V} K(v,w)\\[1.5ex]
&\sum_{v \in V} K(v,v) = 1,\\[1.5ex]
& K(v,w) = 0\quad\text{whenever $\{v,w\} \in E$},\\[1ex]
&\text{$K \in \R^{V \times V}$ is positive semidefinite}.
\end{array}
\end{equation}
Lov\'asz~\cite{Lovasz} proved the sandwich theorem
\begin{equation*}
\alpha(G) \leq \vartheta(G) \leq \chi(\overline{G}),
\end{equation*}
where $\overline{G}$ is the complement of $G$. Sometimes,
$\chi(\overline{G})$ is called the \emph{clique cover number} of
$G$. He also pointed out the relation of the $\vartheta$-number to
Hoffman's bound for the chromatic number, namely that
\begin{equation*}
\chi(G) \geq \vartheta(\overline{G}) \geq \frac{M(A) - m(A)}{-m(A)},
\end{equation*}
by proving that~$\vartheta(\overline{G})$ is the maximum of
\begin{equation*}
\frac{M(B) - m(B)}{-m(B)},
\end{equation*}
where~$B$ ranges over all symmetric matrices in~$\R^{V \times V}$ such
that~$B(v, w) = 0$ whenever~$\{v, w\} \notin E$.

Lov\'asz \cite[Theorem 9]{Lovasz} gave the following spectral bound
for the independence number of an $r$-regular graph:
\begin{equation*}
\alpha(G) \leq \vartheta(G) \leq n \frac{-m(A)}{M(A) - m(A)},
\end{equation*}
where the second inequality is an equality in the case of
edge-transitive graphs. Notice that $r$-regularity of a graph is equivalent
to the property that the all-one vector $1_V = (1, \ldots, 1)^{\sf T}$ is an
eigenvector of the adjacency matrix with eigenvalue~$r$. In this case
$r$ is the largest eigenvalue.

The $\vartheta$-number is actually a lower bound for the
\emph{fractional chromatic number} $\chi^*(\overline{G})$ of the graph
$\overline{G}$, also called the \emph{fractional clique covering
  number} of $G$. This is the minimal value of the sum $\lambda_1 + \cdots
+ \lambda_k$ so that $\lambda_1, \ldots, \lambda_k$ are nonnegative numbers such that
there are independent sets $C_1, \ldots, C_k$ in $\overline{G}$ with
\[
\lambda_1 1_{C_1} + \cdots + \lambda_k 1_{C_k} = 1_V,
\]
where $1_S \in \R^V$ denotes the characteristic vector of
the set $S \subseteq V$.  We have $\chi(\overline{G}) \geq
\chi^*(\overline{G}) \geq \vartheta(G)$.

Schrijver~\cite{Schrijver1979a} introduced a variant of the
$\vartheta$-number called the $\vartheta'$-number by adding
to~\eqref{eq:sdp} the constraint that~$K$ has only nonnegative
entries.  Galtman \cite{Galtman} noticed that
$\vartheta'(\overline{G})$ is related to Hoffman's bound, being the
maximum of
\begin{equation*}
\frac{M(B) - m(B)}{-m(B)},
\end{equation*}
where~$B$ ranges over all symmetric \emph{nonnegative} matrices
in~$\R^{V \times V}$ such that $B(v, w) = 0$ whenever~$\{v, w\} \notin
E$. The proof of this result uses Lov\'asz's original argument
together with Perron-Frobenius theory. One way to interpret this result is
that $\vartheta'$ provides the best spectral bound for all weighted
adjacency matrices where the weights are given by a probability
distribution on the edge set of $G$.

\subsection{Some Hilbert space theory}

We recall some definitions and facts from Hilbert space theory. For
background we refer the reader to the Hilbert space problem
book~\cite{Halmos} by Halmos.

Let $(V, \Sigma, \mu)$ be a measure space consisting of a set~$V$, a
$\sigma$-algebra $\Sigma$ on $V$, and a measure~$\mu$. We consider the
Hilbert space
\begin{equation*}
L^2(V) = \Big\{\,f \colon V \to \C : \text{$f$ measurable, $\int_V |f|^2\, d\mu < \infty$\,}\Big\}
\end{equation*}
of complex-valued square-integrable functions where we identify two
functions which are equal $\mu$-almost everywhere. The inner product
is defined by
\begin{equation*}
(f,g) = \int_V f(x) \overline{g(x)}\, d\mu(x).
\end{equation*}
A linear operator $A \colon L^2(V) \to
L^2(V)$ is \emph{bounded} if there is a nonnegative real number~$M$ so
that for all $f \in L^2(V)$ the inequality
\begin{equation*}
\|Af\| \leq M\|f\|
\end{equation*}
holds. The infimum of the numbers~$M$ having this property is the
\emph{norm} of the operator~$A$, denoted by~$\|A\|$.

The operator~$A$ is \emph{self-adjoint} if we have
\begin{equation*}
(Af,g) = (f,Ag)
\end{equation*}
for all $f,g \in L^2(V)$. The \emph{numerical range} of~$A$ is defined as
\begin{equation*}
W(A) = \{(\,Af,f) : \|f\| = 1\,\}.
\end{equation*}
If $A$ is self-adjoint, then $(Af, f) = (f, Af) = \overline{(Af, f)}$
for each $f \in L^2(V)$, which implies that the numerical range of $A$
is a subset of $\R$. In this case, the numerical range is always an
interval. Indeed, given $f,g \in L^2(V)$ with $\| f \| = \| g \| = 1$
and $g \neq -f$, define for each $t \in [0,1]$ the element
\[
h_t = \frac{t f + (1-t)g}{\| t f + (1-t)g \|}
\]
of $L^2(V)$. Then $t \mapsto (A h_t, h_t)$ is a continuous function
$[0,1] \to \R$ mapping $0$ to $(Af,f)$ and $1$ to $(Ag,g)$.

If moreover $A$ is bounded, then $W(A)$ is contained in the interval
$[-\|A\|, \|A\|]$, by the Cauchy-Schwarz inequality. In this case, we
denote the endpoints of $W(A)$ by $m(A) = \inf\{\,(Af,f) : \|f\| =
1\,\}$ and $M(A) =~\sup\{\,(Af,f) : \|f\| = 1\,\}$.  The two
endpoints~$m(A)$ and~$M(A)$ may or may not belong to~$W(A)$.

\subsection{The independence ratio of an operator}

In~\eqref{characterization} we saw how to characterize the independent
sets of a finite graph in terms of the adjacency matrix. This
motivates the following definition.

\begin{definition}
\label{def:independence}
Let $A \colon L^2(V) \to L^2(V)$ be a bounded, self-adjoint operator. A
measurable set $I \subseteq V$ is called an \emph{independent set} of
$A$ if $(Af,f) = 0$ for each $f \in L^2(V)$ which vanishes almost everywhere
outside of $I$.
\end{definition}

Notice that sets of $\mu$-measure zero are always independent.
In a similar vein, if $I$ is a measurable set and $N$ has measure 
zero, then $I$ is independent if and only if $I \cup N$ is also independent. This says
that we may speak of a measurable set being independent even if we only know
it up to nullsets.

In order to measure the size of an independent set it is convenient to
assume that the measure $\mu$ is a probability measure, although we
will extend these ideas to the Lebesgue upper density of subsets of
$\R^n$ in Section~\ref{sec:euclidean}.

The \emph{independence ratio} of a bounded, self-adjoint
operator~$A\colon L^2(V) \to L^2(V)$ is defined as
\[
\overline{\alpha}(A)= \sup\{\,\mu(I) : I \text{ independent set of }
A\,\}.
\]
In the case of a finite graph~$G = (V, E)$ with adjacency matrix~$A$,
the independence ratio of~$A$ is equal to~$\alpha(G) / |V|$.

Let~$1_V$ denote the all-one function in~$L^2(V)$.  The following
theorem can be used to upper bound the independence ratio of an
operator.

\begin{theorem}
\label{th:alpha_bound}
Let $(V,\mu)$ be a probability space and let $A\colon L^2(V) \to
L^2(V)$ be a nonzero, bounded, self-adjoint operator. Fix a real
number~$R$ and set~$\varepsilon = \|A 1_V - R 1_V\|$.  Suppose there
exists a set $I \subseteq V$ with $\mu(I) > 0$ which is independent
for $A$. Then, if~$R - m(A) - \varepsilon > 0$, we have
\[
\overline{\alpha}(A) \leq \frac{-m(A) + 2\varepsilon}{R - m(A) - \varepsilon}.
\]
\end{theorem}

When~$A$ is the adjacency matrix of a finite regular graph (here $\mu$ is
the uniform probability measure on~$V$), then~$1_V$ is the eigenvector
corresponding to the maximum eigenvalue~$M(A)$. Then, taking~$R =
M(A)$, we have~$\varepsilon = 0$, and we recover~\eqref{eq:alphabound}
from the theorem.

\begin{proof}
Let $I \subseteq V$ be an independent set with $\mu(I) > 0$ and let
$1_I$ denote its characteristic function. Decompose $1_I$ orthogonally as
$\beta 1_V + g$, with $\beta$ a scalar and $g$ orthogonal to~$1_V$.
Notice the identities
\[
\mu(I) = (1_I, 1_V) = \beta
\]
and
\[
\mu(I) = (1_I, 1_I) = \beta^2 + \| g\|^2.
\]
By independence and self-adjointness one therefore has
\[
\begin{split}
0 = (A 1_I, 1_I) & = (A(\beta 1_V + g), \beta 1_V + g)\\
&  = \beta^2 (A1_V,1_V) + \beta (A1_V, g) + \beta(g, A1_V) + (Ag,g).
\end{split}
\]

Let~$\eta = A 1_V - R 1_V$. We consider the first three summands:
\[
(A1_V, 1_V) = (R1_V + \eta, 1_V) = R + (\eta, 1_V),
\]
and
\[
\begin{split}
(A1_V,g) & = (R1_V + \eta, 1_I - \beta 1_V) = R(1_V, 1_I) - R\beta(1_V,1_V) + (\eta,1_I) - \beta(\eta,1_V)\\
& = (\eta,1_I) - \beta(\eta, 1_V),
\end{split}
\]
and similarly
\[
(g,A1_V) = (1_I, \eta) - \beta(1_V, \eta)= (1_I, \eta) - \beta(\eta, 1_V).
\]

Putting it together, using Cauchy-Schwarz, the inequality $\beta \leq
1$, and $(g,g) = \beta-\beta^2$, yields
\[
\begin{split}
0 &=  \beta^2 (A1_V,1_V) + \beta (A1_V, g) + \beta(g, A1_V) + (Ag,g)\\
&=  \beta^2(R + (\eta,1_V)) + \beta((\eta,1_I) + (1_I, \eta) - 2\beta(\eta,1_V)) + (Ag,g)\\
&=  \beta^2 R - \beta^2 (\eta,1_V) + \beta((\eta,1_I) + (1_I, \eta)) + (Ag,g)\\
&\geq \beta^2 R - \beta^2 \varepsilon - 2 \beta \varepsilon + m(A) (\beta-\beta^2).
\end{split} 
\]
Now we divide by $\beta$ and rearrange the terms and get the desired result
\[
\frac{-m(A) + 2\varepsilon}{R- m(A) - \varepsilon} \geq \beta.\qedhere
\]
\end{proof}

Notice that the hypothesis $\mu(I) > 0$ is really needed; if $A$ is the 
operator $L^2(V) \to L^2(V)$ defined by $A 1_V = 1_V$ and $A g = \frac{1}{2} g$
for all $g$ orthogonal to~$1_V$, then clearly the empty set is independent for $A$,
but $-m(A) / (M(A) - m(A)) = -1$.

\subsection{The chromatic number of an operator}
\label{sec:chi-op}

Let $(V, \Sigma, \mu)$ be a measure space. Let $A \colon L^2(V) \to
L^2(V)$ be a bounded, self-adjoint operator. The \emph{chromatic
  number} of~$A$, denoted by~$\chi(A)$, equals the smallest number~$k$
such that one can partition~$V$ into~$k$ independent sets.

We can lower bound the chromatic number of~$A$ when we know the
two endpoints $m(A)$ and $M(A)$ of the numerical range~$W(A)$. This is
completely analogous to Hoffman's bound \eqref{eq:chibound} for finite
graphs when $A$ is the adjacency matrix of the graph.

\begin{theorem}
\label{th:chromatic bound}
Let $A \colon L^2(V) \to L^2(V)$ be a nonzero, bounded, self-adjoint operator.
If $\chi(A) < \infty$, then
\begin{equation*}
\chi(A) \geq \frac{M(A) - m(A)}{-m(A)}.
\end{equation*}
\end{theorem}

\begin{proof}
Let $C_1, \ldots, C_k$ be a partition of~$V$ into independent
sets. Recall that the union of an independent set with a nullset is
also independent. Hence we may assume that~$\mu(C_i) > 0$ for all~$i
= 1$, \dots,~$k$.

We decompose the Hilbert space $L^2(V)$ into an orthogonal direct sum
\begin{equation*}
L^2(V) = \bigoplus_{i=1}^k L^2(C_i).
\end{equation*}
Fix $\varepsilon > 0$. Let $f \in L^2(V)$ be such that $\|f\| = 1$ and
such that $(Af,f) \geq M(A) - \varepsilon$. Decompose $f$ as
\begin{equation*}
f = \sum_{i=1}^k \alpha_i f_i,\quad\text{with $f_i \in L^2(C_i)$ and $\|f_i\| = 1$},
\end{equation*}
and consider the $k$-dimensional Euclidean space $U \subseteq L^2(V)$
with orthonormal basis $f_1, \ldots, f_k$. By $P\colon L^2(V) \to U $ we
denote the orthogonal projection of~$L^2(V)$ onto~$U$. We will
consider the finite-dimensional self-adjoint operator $B\colon U \to U$
defined by $B = P A$. The numerical range of~$B$, in this case the
interval between the smallest and largest eigenvalues, lies in $[m(A),
M(A)]$ because for a unit vector $u \in U$ we have
\begin{equation*}
(Bu,u) = (PAu,u) = (Au,u) \in [m(A), M(A)].
\end{equation*}
Furthermore, the largest eigenvalue~$\lambda_1(B)$ of~$B$ is at least
$M(A)-\varepsilon$ because
\begin{equation*}
\lambda_1(B) \geq (Bf,f) = (Af,f) \geq M(A) - \varepsilon.
\end{equation*}
The trace of~$B$ equals zero because
\begin{equation*}
\trace B = \sum_{i=1}^k (Bf_i, f_i) = \sum_{i=1}^k (Af_i, f_i) = 0,
\end{equation*}
where $(Af_i, f_i) = 0$ since $f_i \in L^2(C_i)$. Now for the sum of the
eigenvalues of $B$ the following holds:
\begin{equation*}
M(A)-\varepsilon + (k-1) m(A) \leq \sum_{i = 1}^k \lambda_i(B) = \trace B = 0,
\end{equation*}
and hence
\begin{equation*}
k \geq 1 - \frac{M(A)-\varepsilon}{m(A)}.
\end{equation*}
The last inequality holds for all $\varepsilon > 0$ and so the theorem
follows.
\end{proof}

From the proof it follows that if $A \neq 0$ and if $\chi(A) <
\infty$, then $m(A) < 0$ and $M(A) > 0$, because $[m(B), M(B)] \subseteq
[m(A), M(A)]$ and $m(B) < 0$ and $M(B) > 0$ since $\trace B = 0$ and
$B \neq 0$.

Finally, we remark that the above proof is very close to the proof of
Hoffman's bound in the book \cite[Chapter VIII.2, Theorem 7, page
265]{Bollobas} by Bollob\'as. In fact, we only had to include the
epsilon.

\subsection{The fractional chromatic number of an operator}

When $(V,\mu)$ is a probability space we can give a bound for the
fractional chromatic number of an operator.  Let $A\colon L^2(V) \to
L^2(V)$ be a bounded, self-adjoint operator. The \emph{fractional
  chromatic number} of~$A$, denoted by $\chi^*(A)$, is the infimum
over all sums $\lambda_1 + \cdots + \lambda_k$ of nonnegative numbers
$\lambda_1$, \dots,~$\lambda_k$ such that there are independent sets
$C_1$, \dots,~$C_k$ of $A$ with
\[
\lambda_1 1_{C_1} + \cdots + \lambda_k 1_{C_k} = 1_V.
\]

\begin{theorem}
  Let $(V,\mu)$ be a probability space and let $A\colon L^2(V) \to L^2(V)$
  be a bounded, self-adjoint operator which is not zero. If $\chi^*(A) <
  \infty$, then
\begin{equation*}
\chi^*(A) \geq \frac{(A 1_V, 1_V) - m(A)}{-m(A)}.
\end{equation*}
\end{theorem}

\begin{proof}
Consider a fractional coloring $\lambda_1 1_{C_1} + \cdots + \lambda_k 1_{C_k} =
1_V$ with $\gamma = \lambda_1 + \cdots + \lambda_k$. Then,
\[
\begin{split}
0 \leq & \sum_{i=1}^k \lambda_i ((A-m(A)I) (\gamma 1_{C_i} - 1_V),\gamma
1_{C_i} - 1_V)\\
= & \gamma^2 \sum_{i=1}^k \lambda_i ((A-m(A)I) 1_{C_i}, 1_{C_i}) - 2\gamma
\sum_{i=1}^k \lambda_i ((A-m(A)I) 1_{C_i}, 1_V)\\
&\qquad + \sum_{i=1}^k \lambda_i ((A-m(A)I) 1_V, 1_V)\\
= & -\gamma^2 m(A) - \gamma ((A-m(A)I) 1_V, 1_V)\\
= & -\gamma^2 m(A) - \gamma ((A 1_V, 1_V) - m(A)),
\end{split}
\]
and so
\[
\gamma \geq \frac{(A 1_V, 1_V) - m(A)}{-m(A)}.\qedhere
\]
\end{proof}

The above proof is very close to the proof of the sandwich theorem
given in Schrijver \cite[Theorem 67.1]{Schrijver2003}.

\subsection{Relation to the $\vartheta$-number}

In Section~\ref{sec:graph-theory} we quickly discussed the relation
between Hoffman's bound and the Lov\'asz $\vartheta$-number. We now
attempt to develop an analogous theory that relates the spectral
bounds we presented for the chromatic number and independence ratio of
operators with suitable generalizations of the Lov\'asz
$\vartheta$-number for some classes of infinite graphs. This theory is
based on transfering arguments, mainly due Lov\'asz, from the finite
into our infinite setting. So in a sense, our contribution here is to
come up with appropriate definitions.

In what follows we shall work with measurable graphs. Let~$V$
be a topological space and~$\mu$ be a Borel probability measure
on~$V$. A graph~$G = (V, E)$ is \emph{measurable} if~$E$ is measurable
as a subset of the product space~$V \times V$.

The \emph{independence ratio} of~$G$ is defined as
\[
\overline{\alpha}(G) = \sup\{\, \mu(I) : \text{$I$ a measurable
  independent set of~$G$}\, \}.
\]

The \emph{measurable chromatic number} of~$G$, denoted by~$\chim(G)$,
is the minimum~$k$ such that~$V$ can be partitioned into~$k$
measurable independent sets.

We say that an operator~$A\colon L^2(V) \to L^2(V)$
\emph{respects}~$G$ if measurable independent sets of~$G$ are also
independent sets of~$A$. Notice that in this case
\[
\overline{\alpha}(G) \leq \overline{\alpha}(A)\qquad\text{and}\qquad
\chi(A) \leq \chim(G).
\]

So from any operator that respects~$G$ we may obtain bounds for the
independence ratio and the measurable chromatic number of~$G$. In
fact, as long as~$\overline{\alpha}(G) > 0$, we get from
Theorem~\ref{th:alpha_bound} that
\begin{equation}
\label{eq:alpha-best}
\overline{\alpha}(G) \leq \inf_A \frac{-m(A)}{M_1(A) -
  m(A)},
\end{equation}
where the infimum is taken over all nonzero, bounded, and self-adjoint
operators~$A\colon L^2(V) \to L^2(V)$ that respect~$G$ and are such
that~$A 1_V = M_1(A) 1_V$ and~$M_1(A) - m(A) > 0$. (Here, we take~$R =
(A 1_V, 1_V)$ and hence~$\varepsilon = 0$ in the statement of the theorem.)

Similarly, from Theorem~\ref{th:chromatic bound} we see
that, if~$\chim(G) < \infty$, then
\begin{equation}
\label{eq:chi-best}
\chim(G) \geq \sup_A \frac{M(A) - m(A)}{-m(A)},
\end{equation}
where the supremum is taken over all nonzero, bounded, and
self-adjoint operators~$A\colon L^2(V) \to L^2(V)$ that respect~$G$.

For some vertex-transitive measurable graphs we will see that both the
infimum in~\eqref{eq:alpha-best} and the supremum
in~\eqref{eq:chi-best} correspond to natural generalizations of the
Lov\'asz $\vartheta$-number to measurable graphs. We will also show
that the product of the infimum and the supremum in this case
equals~$1$, showing that this property of the $\vartheta$-number of
finite graphs carries over to this setting.

An \emph{automorphism} of~$G$ is a measure preserving
bijection~$\varphi\colon V \to V$ (i.e., both~$\varphi$ and its
inverse are measure preserving) that preserves the adjacency
relation, that is, $\{\varphi(v), \varphi(w)\} \in E$ if and only
if~$\{v, w\} \in E$.

We say that~$G$ is \emph{vertex-transitive} if there is a
subgroup~$\Tcal$ of the automorphism group of~$G$ that is a
topological group, acts continuously on~$V$ (i.e., the ``action
map''~$(T, v) \mapsto T \cdot v$ is a continuous function), and acts
transitively on~$V$. We call any such group~$\Tcal$ a
\emph{transitivity group} of~$G$.

In what follows, if~$A\colon L^2(V) \to L^2(V)$ is a bounded,
self-adjoint operator, we write~$A \succeq 0$ if~$A$ is a
\emph{positive operator}, that is, if $(A f, f) \geq 0$ for all~$f \in
L^2(V)$, or equivalently, if the numerical range of~$A$ is
nonnegative. We will denote by~$I$ the identity operator and
by~$J\colon L^2(V) \to L^2(V)$ the Hilbert-Schmidt operator such that
\[
J f = (f, 1_V) 1_V
\]
for all~$f \in L^2(V)$.

\begin{theorem}
\label{thm:gen-theta}
Let~$G = (V, E)$ be a measurable graph with positive independence
ratio.  The infimum in~\eqref{eq:alpha-best} is at least
\begin{equation}
\label{eq:theta-gen}
\begin{array}{rl}
\inf&\lambda\\
&\lambda I + Z - J \succeq 0,\\
&\text{$Z\colon L^2(V) \to L^2(V)$ is a bounded, self-adjoint}\\
&\qquad\text{operator that respects~$G$,}
\end{array}
\end{equation}
with equality when~$G$ is vertex-transitive with a compact
transitivity group and the infimum above is~$< 1$.
\end{theorem}

When~$G$ is a finite graph, problem~\eqref{eq:theta-gen} is one of the
formulations for~$\vartheta(G)$. Namely, it is the dual of the
semidefinite programming problem~\eqref{eq:sdp}.

To prove the second part of the theorem, namely that when~$G$ has a
compact transitivity group, then we have equality, we will need to
symmetrize an operator with respect to the transitivity group. This
operation will be used again later, so we present it now.

Suppose the measurable graph~$G = (V, E)$ has a compact transitivity
group. Let~$\Tcal$ be such a group, and denote by~$\nu$ the Haar
measure on~$\Tcal$, normalized such that~$\nu(\Tcal) = 1$. For~$T \in
\Tcal$, denote the right action on $L^2(V)$ by
\[
(f^T)(x) = f(T \cdot x) \quad \text{with } f \in L^2(V).
\]
If~$A\colon L^2(V) \to L^2(V)$ is a bounded operator, we
denote by~$\Rcal_\Tcal(A)$ the operator such that
\[
(\Rcal_\Tcal(A) f)(x) = \int_\Tcal (Af^T)^{T^{-1}}(x) \, d\nu(T).
\]
Note~$\Rcal_\Tcal(A)$ is a bounded operator, and that~$\Rcal_\Tcal(A)$
is self-adjoint if~$A$ is self-adjoint. Note also that
$\Rcal_\Tcal(A)$ is a positive operator when $A$ is.

Two properties we use of the symmetrized operator are the
following. First, if~$A$ respects~$G$, then also
does~$\Rcal_\Tcal(A)$, as can be easily seen from the fact
that~$\Tcal$ is a subgroup of the automorphism group
of~$G$. Second,~$\Rcal_\Tcal(A)$ always has~$1_V$ as an
eigenfunction. Indeed, if~$x$, $y \in V$, then for some~$U \in \Tcal$
we have~$x = U \cdot y$. Then, using the invariance of the Haar measure
and using~$1^U_V = 1_V$,
\[
\begin{split}
(\Rcal_\Tcal(A) 1_V)(x) &= (\Rcal_\Tcal(A) 1_V)(U \cdot y)\\
&= \int_\Tcal (A 1_V^T)(T^{-1} U \cdot y)\, d\nu(T)\\
&= \int_\Tcal (A (1_V^U)^T)(T^{-1} \cdot y)\, d\nu(T)\\
&= \int_\Tcal (A 1_V^T)(T^{-1} \cdot y)\, d\nu(T)\\
&= (\Rcal_\Tcal(A) 1_V)(y),
\end{split}
\]
as we wanted.

\begin{proof}[Proof of Theorem~\ref{thm:gen-theta}]
Let~$A\colon L^2(V) \to L^2(V)$ be a bounded, self-adjoint operator
that respects~$G$. Assume moreover~$A 1_V = M_1(A) 1_V$ and~$M_1(A) -
m(A) > 0$. Since
\[
0 < \overline{\alpha}(G) \leq \overline{\alpha}(A) \leq
\frac{-m(A)}{M_1(A) - m(A)},
\]
we see that~$m(A) < 0$. So we may assume that~$m(A) = -1$.

We claim that~$\lambda = (M_1(A) - m(A))^{-1}$ and~$Z = \lambda A$ form
a feasible solution of problem~\eqref{eq:theta-gen}. Obviously,~$Z$
respects $G$. We show that~$\lambda I + Z - J \succeq 0$.

For this, write~$X = \lambda I + Z$. Since~$m(A) = -1$, we have~$X
\succeq 0$. Now take~$f \in L^2(V)$ and say~$f = \beta 1_V + w$ for
some~$w$ orthogonal to~$1_V$. Then
\[
\begin{split}
((\lambda I + Z - J) f, f) &= (X(\beta 1_V + w), \beta 1_V + w) -
(J(\beta 1_V + w), \beta 1_V + w)\\
&= |\beta|^2 (X 1_V, 1_V) + \Real(2\beta (X 1_V, w)) + (Xw, w)\\
&\qquad{} - (|\beta|^2 (J 1_V, 1_V) + \Real(2\beta (J 1_V, w)) + (J w, w))\\
&=|\beta|^2 + (Xw, w) - |\beta|^2\\
&\geq 0.
\end{split}
\]
Here, we used the fact that~$1_V$ is an eigenfunction of~$X$ with
eigenvalue~$1$ and that~$(1_V, w) = 0$. We also used that~$X \succeq
0$, since then~$(Xw, w) \geq 0$.

So we have that the optimal value of~\eqref{eq:theta-gen} is at
most~$\lambda = -m(A) / (M_1(A) -m(A))$ since~$m(A) = -1$, proving the
inequality we wanted.

Now we show that equality holds when~$G$ is vertex-transitive. We
start by observing that, in general, any feasible solution
of~\eqref{eq:theta-gen} gives an upper bound
to~$\overline{\alpha}(G)$. Indeed, let~$I$ be a measurable independent
set of~$G$ with~$\mu(I) > 0$. Then
\[
0 \leq ((\lambda I + Z - J) 1_I, 1_I) = \lambda \mu(I) - \mu(I)^2,
\]
and dividing by~$\mu(I)$ we obtain~$\mu(I) \leq \lambda$.

So let~$\lambda < 1$ and~$Z$ be a feasible solution
of~\eqref{eq:theta-gen}. From the above observation, we have~$\lambda
> 0$. Since~$\lambda < 1$, we know that~$Z$ is nonzero.  We may also
assume that~$1_V$ is an eigenfunction of~$Z$, for if not then we
replace $Z$ with $\Rcal_\Tcal(Z)$, where~$\Tcal$ is a compact
transitivity group of~$G$. Then
\[
\Rcal_\Tcal(\lambda I + Z - J) = \lambda I + \Rcal_\Tcal(Z) - J
\succeq 0
\]
implies that $\lambda$ and $\Rcal_\Tcal(Z)$ are also feasible for~\eqref{eq:theta-gen}
and~$\Rcal_\Tcal(Z)$ has~$1_V$ as an eigenfunction.

Then~$Z$ is a nonzero, bounded, self-adjoint operator that
respects~$G$. If~$Z 1_V = M_1(Z) 1_V$, then we have
\[
M_1(Z) = (Z 1_V, 1_V) \geq (J 1_V, 1_V) - \lambda (I 1_V, 1_V) = 1 -
\lambda.
\]
So we see that~$M_1(Z) > 0$, as~$\lambda < 1$.

We also have~$m(Z) \geq -\lambda$. Indeed, let~$f \in L^2(V)$ with
$\|f\| \leq 1$ and
write~$f = \alpha 1_V + w$, where~$w$ is orthogonal to~$1_V$. Then,
since~$1_V$ is an eigenfunction of~$Z$, we get
\[
(Z f, f) = (Z (\alpha 1_V + w), \alpha 1_V + w) = \alpha^2 (Z 1_V,
1_V) + (Z w, w) \geq (Z w, w),
\]
and we get from~$\lambda I + Z - J \succeq 0$ that
\[
(Zf, f) \geq (Zw, w) \geq -\lambda (w, w) \geq -\lambda.
\]

Notice that we must have~$m(Z) < 0$. Indeed, if~$m(Z) > 0$,
then~$(Z 1_I, 1_I) > 0$ for any measurable independent set~$I$ of~$G$
with~$\mu(I) > 0$, a contradiction since~$Z$ respects~$G$. If~$m(Z) =
0$, then~$M_1(Z) - m(Z) = M_1(Z) > 0$, and we have
by~\eqref{eq:alpha-best} that~$\overline{\alpha}(G) \leq 0$, a
contradiction since we assume~$G$ has positive independence ratio.

Hence, since~$M_1(Z) \geq 1 - \lambda$ and~$0 < -m(Z) \leq \lambda$ we
have
\[
\frac{M_1(Z) - m(Z)}{-m(Z)} = 1 + \frac{M_1(Z)}{-m(Z)} \geq
\frac{1}{\lambda},
\]
showing that the infimum in~\eqref{eq:alpha-best} is at
most~$\lambda$.
\end{proof}

Given~$a \in L^\infty(V)$, we denote by~$D_a\colon L^2(V) \to L^2(V)$
the multiplication operator~$(D_a f)(x) = a(x) f(x)$ for~$f \in
L^2(V)$. Note~$D_a$ is a bounded operator. The following theorem
connects~\eqref{eq:chi-best} with the Lov\'asz $\vartheta$-number.

\begin{theorem}
\label{thm:gen-thetabar}
Let~$G = (V, E)$ be a measurable graph with finite measurable
chromatic number. The supremum in~\eqref{eq:chi-best} is at most
\begin{equation}
\label{eq:thetabar-gen}
\begin{array}{rl}
\sup&((D_a + K) 1_V, 1_V)\\
&D_a + K \succeq 0,\\
&\text{$a \in L^\infty(V)$ and~$\int_V a(x)\, d\mu(x) =
  1$},\\
&\text{$K\colon L^2(V) \to L^2(V)$ is a bounded, self-adjoint}\\
&\qquad\text{operator that respects~$G$,}
\end{array}
\end{equation}
with equality when~$G$ is vertex-transitive with a compact
transitivity group and the supremum above is~$> 1$.
\end{theorem}

Again, when~$G$ is a finite graph,
problem~\eqref{eq:thetabar-gen} corresponds to the semidefinite
programming problem~\eqref{eq:sdp} applied to the complement
of~$G$. So for a finite graph~$G$ the optimization problem above gives
us~$\vartheta(\overline{G})$.

\begin{proof} 
Let~$A\colon L^2(V) \to L^2(V)$ be a nonzero, bounded, self-adjoint
operator that respects~$G$. Recall from Section~\ref{sec:chi-op}
that~$M(A) > 0 > m(A)$. So we may assume that~$m(A) = -1$.

Fix~$\varepsilon > 0$ and let~$f \in L^\infty(V)$ be such that~$(A f,
f) > M(A) - \varepsilon$ and~$\|f\| = 1$ (such an~$f$ exists,
because~$L^\infty(V)$ is dense in~$L^2(V)$ and~$A$ is
bounded). Let~$a(x) = |f(x)|^2$ and set~$K = D_f^* A D_f$. Notice~$K$
is a bounded and self-adjoint operator and~$\int_V a(x)\, d\mu(x)
= 1$. Also, by construction~$K$ respects~$G$ and~$D_a + K \succeq 0$,
since~$D_a + K = D_f^* (I + A) D_f$ and~$I + A \succeq 0$ as~$m(A) =
-1$.

So~$a$ and~$K$ form a feasible solution
of~\eqref{eq:thetabar-gen}. Moreover
\[
\begin{split}
((D_a + K) 1_V, 1_V) &= (D_a 1_V, 1_V) + (D_f^* A D_f 1_V, 1_V)\\
&= 1 + (A D_f 1_V, D_f 1_V)\\
&= 1 + (A f, f)\\
&\geq 1 + M(A) - \varepsilon,
\end{split}
\]
and taking~$\varepsilon \to 0$ we obtain that the optimal
of~\eqref{eq:thetabar-gen} is at least
\[
1 + M(A) = \frac{M(A) - m(A)}{-m(A)},
\]
as we wanted.

Now we show that when~$G$ is vertex-transitive, then equality holds,
as long as the optimal value of~\eqref{eq:thetabar-gen} is~$>
1$. Indeed, let~$a$ and~$K$ be a feasible solution
of~\eqref{eq:thetabar-gen} with~$K$ nonzero. Such solution must exist,
since the optimal value of our problem is greater than~$1$.

Then~$K$ is a nonzero, bounded, self-adjoint operator that
respects~$G$. We may assume that~$a = 1_V$, so that~$D_a = I$. Indeed,
let~$\Tcal$ be a compact transitivity group of~$G$. Then~$\overline{a}
= 1_V$ and~$\overline{K} = \Rcal_\Tcal(K)$ form a feasible solution
of~\eqref{eq:thetabar-gen} (note~$\Rcal_\Tcal(D_a) = I$, since~$G$ is
vertex transitive). Also,~$((D_{\overline{a}} + \overline{K}) 1_V,
1_V) = ((D_a + K) 1_V, 1_V)$, so we may take the symmetrized solution
instead of the original one.

Then, since~$a = 1_V$, we have~$(D_a f, f) \leq 1$ for all~$f \in
L^2(V)$ with $\|f\| \leq 1$. This, together with~$D_a + K \succeq 0$, implies that~$m(K)
\geq -1$. Then since~$M(K) \geq (K 1_V, 1_V)$, we have
\[
\frac{M(K) - m(K)}{-m(K)} \geq 1 + M(K) \geq 1 + (K 1_V, 1_V) = ((D_a
+ K) 1_V, 1_V),
\]
proving the equality.
\end{proof}

When~$G = (V, E)$ is a finite graph, we have~$\vartheta(G)
\vartheta(\overline{G}) \geq |V|$, with equality when~$G$ is
vertex-transitive. We finish this section by observing that the same
holds for the optimal values of~\eqref{eq:theta-gen}
and~\eqref{eq:thetabar-gen}, showing that this property carries on to
our setting.

\begin{theorem}
\label{th:transitive}
Let~$G = (V, E)$ be a measurable graph with positive independence
ratio and finite measurable chromatic number. Let~$\theta$ be the optimal
value of~\eqref{eq:theta-gen} and~$\tilde{\theta}$ be the optimal
value of~\eqref{eq:thetabar-gen}. Then~$\theta \cdot \tilde{\theta} \geq
1$, with equality when~$G$ has a compact transitivity group.
\end{theorem}

\begin{proof}
Let~$\lambda$ and~$Z$ be a feasible solution
of~\eqref{eq:theta-gen}. Then~$\lambda > 0$, since~$G$ has a positive
independence ratio. Set~$a = 1_V$, so that~$D_a = I$, and~$K =
\lambda^{-1} Z$. Then we have that~$K$ respects~$G$ and that~$D_a + K
\succeq 0$, so~$a$ and~$K$ are a feasible solution
of~\eqref{eq:thetabar-gen}. Moreover
\[
\begin{split}
((D_a + K) 1_V, 1_V) &= ((I + \lambda^{-1} Z - \lambda^{-1} J) 1_V,
1_V) + (\lambda^{-1} J 1_V, 1_V)\\
&\geq \lambda^{-1} (J 1_V, 1_V)\\
&= \lambda^{-1},
\end{split}
\]
so that~$\tilde{\theta} \geq \lambda^{-1}$, and we see that~$\theta
\cdot \tilde{\theta} \geq 1$.

To see the reverse inequality when~$G$ has a compact transitivity
group~$\Tcal$, let~$a$ and~$K$ be a feasible solution
of~\eqref{eq:thetabar-gen}. Notice~$\overline{a} = 1_V$
and~$\overline{K} = \Rcal_\Tcal(K)$ are also feasible
for~\eqref{eq:thetabar-gen}, and~$((D_{\overline{a}} + \overline{K})
  1_V, 1_V) = ((D_a + K) 1_V, 1_V)$. 

Notice~$D_{\overline{a}} = I$. Set~$\beta = ((I + \overline{K}) 1_V,
1_V)$. We claim~$\lambda = \beta^{-1}$ and~$Z = \beta^{-1} \overline{K}$
form a feasible solution for~\eqref{eq:theta-gen}.

Indeed, notice~$Z$ respects~$G$ by construction. We show that~$\lambda
I + Z - J \succeq 0$. For this, write~$X = I + \overline{K}$ and
let~$f \in L^2(V)$. Write~$f = \alpha 1_V +
w$, where~$w$ is orthogonal to~$f$. Then
\[
\begin{split}
((\lambda I + Z - J) f, f) &= \beta^{-1} (Xf, f) - (Jf, f)\\
&=\beta^{-1} (X(\alpha 1_V + w), \alpha 1_V + w) - (J(\alpha 1_V + w),
\alpha 1_V + w)\\
&=\beta^{-1} (\alpha^2 (X 1_V, 1_V) + 2\alpha(X 1_V, w) + (X w, w))\\
&\qquad{}-(\alpha^2 (J 1_V, 1_V) + 2\alpha(J 1_V, w) + (J w, w))\\
&=\beta^{-1} \alpha^2 \beta + (Xw, w) - \alpha^2\\
&\geq 0,
\end{split}
\]
where we use the fact that~$X \succeq 0$, proving the claim.

So we see that~$\theta \leq \lambda = \beta^{-1}$, and so~$\theta
\cdot \tilde{\theta} \leq 1$, as we wanted.
\end{proof}

\section{Graphs on Euclidean space}
\label{sec:euclidean}

In this section we consider translation invariant graphs defined on
the Euclidean space~$V = \R^n$. Let $N \subseteq \R^n$ be a bounded,
Lebesgue measurable set which does not contain the origin in its
topological closure. Say two vertices $x, y\in \R^n$ are adjacent
whenever $x - y \in N$. To make the adjacency relation symmetric we
require that $N$ is centrally symmetric, i.e.\ $N = -N$. We denote
this graph by~$G(\R^n, N)$.

Our aim is to determine lower bounds for the measurable chromatic
number $\chim(G(\R^n, N))$ of this graph. This is the smallest number
of colors one needs to paint all points of~$\R^n$ so that two points
which are adjacent receive different colors and all points having the
same color form measurable sets. This number is finite since $N$ is
bounded and since every point is in an open set of positive measure
which does not contain any of its neighbors in the graph.

In general, finding $\chim(G(\R^n, N))$ is a notoriously difficult
problem: It has been intensively studied for the unit sphere $N =
S^{n-1} = \{x \in \R^n : x \cdot x = 1\}$ and there even in the planar
case the number is only known to lie between five and seven; see
Soifer \cite{Soifer} and Sz\'ekely \cite{Szekely} for the history of
this problem.

The technique we present is a common extension of the technique of
Steinhardt~\cite{Steinhardt} who used it to show that
\begin{equation*}
\lim_{N \to \infty} \chim\left(G\left(\R^2, \bigcup_{k = 0}^N (2k+1)S^1\right)\right) = \infty,
\end{equation*}
and an extension of the technique of Oliveira and Vallentin~\cite{OliveiraVallentin} who gave
an upper bound for the measurable chromatic number of graphs of the
form
\begin{equation*}
G\left(\R^n, d_1 S^{n-1} \cup d_2 S^{n-1} \cup \cdots \cup d_N S^{n-1}\right),
\end{equation*}
which also led to the best known lower bounds for the measurable
chromatic number of $G(\R^n, S^{n-1})$ in dimensions~$3, \ldots,
24$. We revisit these two examples in Section~\ref{ssec:odddistance}
and Section~\ref{ssec:unitdistance}.

\subsection{Computing lower bounds for measurable chromatic numbers via Fourier analysis}

Let $\nu$ be a signed Borel measure (i.e.\ which does not take the
values $\pm \infty$) with support contained in $N$, and which is centrally
symmetric, i.e.\ $\nu(-S) = \nu(S)$ for all measurable sets~$S$. The
convolution operator $A_\nu \colon L^2(\R^n) \to L^2(\R^n)$ given by
\begin{equation}
\label{eq:Anu}
(A_\nu f)(x) = (f * \nu)(x) = \int_{\R^n} f(x-y) d\nu(y),\;\; \text{for $f \in L^2(\R^n)$,}
\end{equation}
is a bounded operator, by Minkowski's integral inequality (see
e.g.~\cite[Proposition 8.49]{Folland}). The fact that $\nu$ is
centrally symmetric implies that $A_{\nu}$ is self-adjoint:
\[
\begin{split}
(A_{\nu} f, g) & =  \int_{\R^n} \int_{\R^n} f(x-y) \, d\nu(y) \,
\overline{g(x)} \, dx = \int_{\R^n} f(x) \int_{\R^n} \overline{g(x+y)}
\, d\nu(y) \, dx\\
 & = \int_{\R^n} f(x) \int_{\R^n} \overline{ g(x-y) } \, d\nu(y) \, dx = (f, A_{\nu} g).
\end{split}
\]

If $I$ is a measurable independent set of~$G(\R^n, N)$ then it is also
an independent set of any such convolution operator. In fact, let $f
\in L^2(\R^n)$ be a function which vanishes almost everywhere outside
of~$I$. For $x \in I$ and $y \in N$ we cannot have $x-y \in I$, and
therefore $f(x-y) \overline{f(x)} = 0$ for Lebesgue almost
every~$x$. Hence,
\begin{equation*}
(A_{\nu}f, f) = \int_{\R^n} \int_{\R^n} f(x - y) \overline{f(x)} \,
d\nu(y) \, dx =  0.
\end{equation*}

So~$A_\nu$ respects~$G(\R^n,N)$ and $\chi(A_\nu) \leq \chim(G(\R^n,N))$.
Theorem~\ref{th:chromatic bound} then gives 
\begin{equation*}
\frac{M(A_\nu) - m(A_\nu)}{-m(A_\nu)} \leq \chi(A_\nu) \leq \chim(G(\R^n, N))
\end{equation*}
for every signed Borel measure supported on~$N$ which is centrally symmetric.

To determine the numerical range of $A_{\nu}$ we apply the Fourier
transform which by Plancherel's theorem is a unitary operator on
$L^2(\R^n)$:
\begin{equation*}
(A_{\nu} f, f) = (\widehat{A_{\nu} f}, \widehat{f}) = (\widehat{f *
  \nu}, \widehat{f}) = (\widehat{\nu} \widehat{f}, \widehat{f}),
\end{equation*}
where 
\[
\widehat{\nu}(u) = \int_{\R^n} e^{-2\pi i x \cdot u} \, d\nu(x)
\]
is the Fourier transform of the measure $\nu$.

So it suffices to determine the numerical range of the multiplication
operator $g \mapsto \widehat{\nu} g$.
If $g \in L^2(\R^n)$ with $\| g \| = 1$, then clearly
\[
(\widehat{\nu} g, g) = \int | g(x) |^2 \widehat{\nu}(x) dx
\]
lies between $\inf_{u \in \R^n} \widehat{\nu}(u)$ and $\sup_{u \in
  \R^n} \widehat{\nu}(u)$, since $\widehat{\nu}$ is continuous and
bounded. Note that $\widehat{\nu}$ is real-valued because $\nu$ is
centrally symmetric.
If $\epsilon > 0$, then choose $x_0 \in \R^n$ with
$\widehat{\nu}(x_0) \leq \inf_{u \in \R^n} \widehat{\nu}(u) + \epsilon$. If $B_r(x_0)$ denotes
the ball of radius $r$ centered at $x_0$ and $\vol(B_r(x_0))$ its
volume, then by 
\[
\lim_{r \to 0} \int_{\R^n} \widehat{\nu}(x) \left| \frac{1_{B_r(x_0)}}{\sqrt{\vol(B_r(x_0))}} \right|^2  dx =
\lim_{r \to 0} \frac{1}{\vol(B_r(x_0))} \int_{B_r(x_0)} \widehat{\nu}(x) dx = \widehat{\nu}(x_0),
\]
whence $m(A_\nu) \leq \widehat{\nu}(x_0) \leq \inf_{u \in \R^n} \widehat{\nu}(u) + \epsilon$,
and since $\epsilon$ was arbitrary we get $m(A_{\nu}) = \inf_{u \in \R^n} \widehat{\nu}(u)$.
A similar argument shows $M(A_\nu) = \sup_{u \in \R^n} \widehat{\nu}(u)$.

So we finally get
\begin{equation}
\label{eq:chibound2}
\frac{\sup_{u \in \R^n} \widehat{\nu}(u) - \inf_{u \in \R^n}
  \widehat{\nu}(u)}{-\inf_{u \in \R^n}\widehat{\nu}(u)} \leq
\chim(G(\R^n, N)).
\end{equation}

To get the best possible bound from this approach, one can optimize
over all measures $\nu$ having the required properties.

\subsection{Computing upper bounds for upper densities via Fourier analysis}

Independent sets in~$G(\R^n, N)$ might have infinite Lebesgue
measure. So, while it does not make sense to look for upper bounds for
the measure of independent sets, we might look for upper bounds for
their upper density, a measure of the fraction of space they
cover. Given a measurable set~$S \subseteq \R^n$, its \emph{upper
  density} is
\[
\overline{\delta}(S) = \limsup_{r \to \infty} \frac{\vol(S \cap
  B_r)}{\vol(B_r)},
\]
where~$B_r$ is the ball of radius~$r$ centered at the origin.

By using the upper density, we may extend the definition of
independence ratio also to the graph~$G(\R^n, N)$. We simply put
\[
\overline{\alpha}(G(\R^n, N)) = \sup\{\, \overline{\delta}(I) :
\text{$I$ independent set of~$G(\R^n, N)$}\,\}.
\]

Now let~$\nu$ be a signed Borel measure with support in~$N$ that is
centrally symmetric. Recall from the previous section that every
independent set of~$G(\R^n, N)$ is also an independent set of the
convolution operator~$A_\nu$ defined in~\eqref{eq:Anu}.

For~$r > 0$, let~$B_r$ be the ball of radius $r$ centered at the
origin. We view $B_r$ as a measure space equipped with the normalized
Lebesgue measure. Denote by $\langle f, g \rangle$ the inner
product in $L^2(B_r)$. Consider the operator~$A_\nu^r\colon L^2(B_r)
\to L^2(B_r)$ given by
\[
A_\nu^r f = (A_\nu f)|_{B_r},
\]
where~$A_\nu f$ is the operator~$A_\nu$ applied to the extension
of~$f$ by zeros to all of~$\R^n$.

Then we have that
\[
\limsup_{r \to \infty} \overline{\alpha}(A_\nu^r) \geq
\overline{\alpha}(G(\R^n, N)).
\]

Using Theorem~\ref{th:alpha_bound} we may upper
bound~$\overline{\alpha}(A_\nu^r)$ for every~$r$. For a given~$r > 0$,
we apply the theorem with~$R = R(r) = \langle A_\nu^r 1_{B_r}, 1_{B_r}
\rangle$ and~$\varepsilon = \varepsilon(r) = \|A_\nu^r 1_{B_r} - R(r)
1_{B_r}\|$. We then obtain
\[
\frac{-m(A_\nu^r) + 2\varepsilon(r)}{R(r) -m(A_\nu^r) - \varepsilon(r)} \geq \overline{\alpha}(A_\nu^r).
\]

We claim:
\[
\lim_{r \to \infty} m(A_\nu^r) = m(A_\nu) = \inf_{u \in \R^n} \widehat{\nu}(u), \quad \lim_{r \to \infty}
R(r) = \widehat{\nu}(0), \quad \text{and } \lim_{r \to \infty} \varepsilon(r) = 0.
\]

Now we prove the first identity. For $f \in L^2(\R^n)$, we write $f^r = f|_{B_r}$. Then,
\[
\langle A_\nu^r f^r, f^r \rangle = \frac{1}{\vol
  B_r} (A_\nu f^r, f^r) \geq \frac{1}{\vol
  B_r} m(A_\nu) (f^r, f^r) = m(A_\nu) \langle f^r, f^r \rangle,
\]
and so $m(A^r_\nu) \geq m(A_\nu)$.

For the reverse inequality pick $f \in L^2(\R^n)$ of unit norm such
that $(A_\nu f, f)$ is close to $m(A_\nu)$. Then, by taking $r$ large
enough we can make the norm of $f^r$ as close to one as we want.
Notice that by Cauchy-Schwarz
\[
m(A^r_\nu) \leq \frac{\langle A^r_\nu f^r, f^r\rangle}{\langle f_r, f_r
  \rangle} = \frac{(A_\nu f^r, f^r)}{(f_r, f_r)} \mathop{\longrightarrow}\limits_{r \to \infty} (A_\nu f, f).
\]

Now we prove the second identity.  Since $\nu$ is supported on a
bounded set, we have
\[
(A^r_\nu 1_{B_r})(x) = \int_{\R^n} 1_{B_r}(x-y) \, d\nu(y) \mathop{\longrightarrow}\limits_{r \to
  \infty} \widehat{\nu}(0),
\]
pointwise in~$x$. Then
\[
\lim_{r \to \infty}\langle A^r_\nu 1_{B_r}, 1_{B_r} \rangle = 
\lim_{r \to \infty} \frac{1}{\vol B_r} \int_{B_r} (A^r_\nu
1_{B_r})(x) \, dx = \widehat{\nu}(0).
\]

Indeed, let~$D$ be the diameter of~$N$. Then if~$(A_\nu^r 1_{B_r})(x)
\neq \widehat{\nu}(0)$, we must have~$x \in B_r \setminus B_{r -
  D}$. Hence
\[
\begin{split}
&\frac{1}{\vol B_r} \int_{B_r} |(A^r_\nu
1_{B_r})(x) - \widehat{\nu}(0)| \, dx\\
&\qquad\leq \frac{1}{\vol B_r} \|A^r_\nu
1_{B_r} - \widehat{\nu}(0) \|_{\infty} \vol(B_r \setminus B_{r - D})\\
&\qquad\leq \frac{\vol(B_r \setminus B_{r - D})}{\vol B_r} \|\nu\| \to 0
\end{split},
\]
as~$r \to \infty$, where $\|\nu\|$ is the total variation norm of $\nu$. 

This argument also proves the third identity in the claim, giving us
\begin{equation}
\label{eq:density-bound}
\frac{-\inf_{u \in \R^n} \widehat{\nu}(u)}{\widehat{\nu}(0) - \inf_{u
    \in \R^n} \widehat{\nu}(u)} \geq \overline{\alpha}(G(\R^n, N)).
\end{equation}

\subsection{The odd distance graph}
\label{ssec:odddistance}

A notoriously difficult problem in discrete geometry is whether the
odd distance graph in the plane
\begin{equation*}
G\left(\R^2, \bigcup_{k=0}^\infty (2k+1) S^1 \right)
\end{equation*}
has a finite chromatic number, a question due to Rosenfeld. Ardal,
Ma\v{n}uch, Rosenfeld, Shelah, and
Stacho~\cite{ArdalManuchRosenfeldShelahStacho} showed that the
chromatic number is at least five. It follows from a theorem of
Furstenberg, Katznelson, and Weiss \cite{FurstenbergKatznelsonWeiss}
that the measurable chromatic number of the odd distance graph is
infinite. Steinhardt~\cite{Steinhardt} gave the following alternative
proof of this fact which can be seen as a nice example of the method
described above.

For a positive number $\beta > 1$ Steinhardt defined the
probability measure
\begin{equation*}
\nu = \frac{\beta-1}{\beta}\sum_{k = 0}^{\infty} \beta^{-k} \omega_{2k+1}
\end{equation*}
on $\R^2$ where $\omega_{2k+1}$ is the rotationally invariant probability measure on
the circle of radius $2k+1$ centered at $0$. Then he showed that
\begin{equation*}
\lim_{\beta \to 1} \inf_{u \in \R^2} \frac{\beta-1}{\beta}\sum_{k =
  0}^{\infty} \beta^{-k} \int_{S^1} e^{-2\pi i x \cdot u} \, d\omega_{2k+1}(x) = 0.
\end{equation*}
When we define the measure with bounded support
\[
\nu_{\beta, N} = \frac{\beta-1}{\beta} \sum_{k = 0}^N \beta^{-k} \omega_{2k+1},
\]
we see that
\[
\sup_{u \in \R^2} \widehat{\nu}_{\beta, N}(u) = \widehat{\nu}_{\beta, N}(0)
    \mathop{\longrightarrow}\limits_{\beta \to 1, N \to \infty} 1,
\]
since $\nu_{\beta, N}$ is positive, and
\[
\inf_{u \in \R^2} \widehat{\nu}_{\beta, N}(u) 
    \mathop{\longrightarrow}\limits_{\beta \to 1, N \to \infty} 0,
\]
and so by \eqref{eq:chibound2}
\[
\lim_{N \to \infty} \chim\left(G\left(\R^2, \bigcup_{k = 0}^N
    (2k+1)S^1\right)\right) \geq \lim_{\beta \to 1} \lim_{N \to
  \infty} \frac{\widehat{\nu}_{\beta,N}(0) - \inf_{u \in \R^2}
  \widehat{\nu}(u)}{-\inf_{u \in \R^2}
  \widehat{\nu}(u)}
= \infty.
\]

In a similar way, \eqref{eq:density-bound} can be used to prove a
quantitative version of the theorem of Furstenberg, Katznelson, Weiss;
see Oliveira and Vallentin \cite[Theorem 5.1]{OliveiraVallentin} and
also Kolountzakis~\cite{Kolountzakis} for a similar result and a
similar proof in the case of arbitrary norms whose unit ball is not
polytopal. To obtain this result Kolountzakis chooses a probablity
measure $\nu$ supported on the boundary of the unit norm ball and
studies the decay of the function $u \to \widehat{\nu}(u)$ when $u$
becomes large.

\subsection{The unit distance graph}
\label{ssec:unitdistance}

Let $N$ be the unit sphere $S^{n-1}$. The orthogonal group acts
transitively on $N$. Then the measure that optimizes the bound in
\eqref{eq:chibound2} for the unit distance graph $G(\R^n, S^{n-1})$ is
the rotationally invariant probability measure $\omega$ on
$S^{n-1}$. Its Fourier transform can be explicitly computed:
\begin{equation*}
\widehat{\omega}(u) = \int_{\R^n} e^{-2\pi i x \cdot u} \, d\omega(x) = \Omega_n(\|u\|),
\end{equation*}
where
\begin{equation*}
\Omega_n(t) = \Gamma\left(\frac{n}{2}\right)
\left(\frac{2}{t}\right)^{(n-2)/2} J_{(n-2)/2}(t) \quad \text{with }  \Omega(0) = 1,
\end{equation*}
and $J_{(n-2)/2}$ is the Bessel function of the first kind with
parameter $(n-2)/2$. The global minimum of $\Omega$ is at $j_{n/2,1}$
which is the first positive zero of the Bessel function
$J_{n/2}$. Hence,
\begin{equation*}
\chim(G(\R^n, S^{n-1})) \geq
\frac{\Omega_n(j_{{n/2},1}) - 1}{\Omega_n(j_{{n/2},1})}
\quad \text{and} \quad
\overline{\alpha}(G(\R^n, S^{n-1})) \leq \frac{\Omega_n(j_{{n/2},1})}{\Omega_n(j_{{n/2},1}) - 1}
\end{equation*}
recovering the result of Oliveira and Vallentin \cite[Section
3]{OliveiraVallentin}.

\section{Graphs on the unit sphere}
\label{sec:unit sphere}

In this section we consider distance graphs defined on the unit sphere
$S^{n-1} = \{x \in \R^n : x \cdot x = 1\}$. To define the edge set we
use a Borel subset $D$ of the interval $[-1,1]$ where $1$ does not lie
in the topological closure of~$D$. Then two vertices are adjacent
whenever $x \cdot y \in D$. We denote this graph
by~$G(S^{n-1},D)$. Again, we aim at lower bounding the measurable
chromatic number~$\chim(G(S^{n-1},D))$. Here we extend the technique
of Bachoc, Nebe, Oliveira and
Vallentin~\cite{BachocNebeOliveiraVallentin}, who gave a formulation
for the $\vartheta$-number of~$G(S^{n-1}, D)$. However, they showed
how to compute it only in the case that~$D$ is finite.

The techniques in this section are quite similar to those presented in
the previous section. The orthogonal group, which is a compact,
non-commutative group, is a transitivity group of the measurable graph
$G(S^{n-1},D)$. So one can interpret the results in this section as
the compact, non-abelian case whereas those in the previous section as
the locally-compact, abelian case. In principle there is no technical
difficulty to extend the results of this section from graphs on the
sphere to graphs on compact, connected, rank-one symmetric spaces, see
Oliveira and Vallentin~\cite{OliveiraVallentin2013a}.

\subsection{Computing spectral bounds via spherical harmonics}

Let $\nu$ be a signed Borel measure which is supported on the
set~$D$. For $t \in (-1,1)$ define the operator $A_t \colon C(S^{n-1})
\to C(S^{n-1})$ by
\[
(A_t f)(\xi) = \int_{S^{n-1}} f(\eta) \, d\omega_{\xi,t}(\eta),
\]
where $\omega_{\xi,t}(\eta)$ is the rotationally invariant probability
measure on the $(n-2)$-dimensional sphere $\{\eta \in S^{n-1} : \xi
\cdot \eta = t\}$. 

We choose an orthonormal basis of $C(S^{n-1})$, and so of
$L^2(S^{n-1})$, consisting of spherical harmonics $S_{k,l}$ where $k =
0, 1, \ldots$ and $l = 1, \ldots, c_{k,n}$ with $c_{k,n} =
\binom{k+n-2}{k} + \binom{k+n-3}{k-1}$. The degree of $S_{k,l}$
equals~$k$.

For $\delta > 0$ we have
\[
(A_t f)(\xi) = 
\lim_{\delta \to 0} \frac{1}{2\delta} \int_{t-\delta}^{t + \delta}
\int_{S^{n-1}} f(\eta) \, d\omega_{\xi,u}(\eta)
\frac{(1-u^2)^{(n-3)/2}}{(1-t^2)^{(n-3)/2}} \, du
\]
and by the Funk-Hecke formula (see \cite[Theorem
9.7.1]{AndrewsAskeyRoy}\footnote{In \cite{AndrewsAskeyRoy} the
  Funk-Hecke formula is only stated for continuous functions $f \in
  C([-1,1])$ but in fact it is also valid when $f \in L^1([-1,1])$,
  see Groemer~\cite[Chapter 3.4]{Groemer1996a}. Notice also that we
  use a more convenient normalization of the surface measure here so
  that their term $\omega_{n-1}$ disappears.}) we see that the
spherical harmonics $S_{k,l}$ are eigenfunctions of $A_t$ with
eigenvalue
\[
\lambda_k(t) = \lim_{\delta \to 0} \frac{1}{2\delta}
\int_{t-\delta}^{t + \delta}
\overline{P}^{(\alpha,\alpha)}_k(u)
\frac{(1-u^2)^{(n-3)/2}}{(1-t^2)^{(n-3)/2}} \, du = 
\overline{P}^{(\alpha,\alpha)}_k(t),
\] 
where~$\alpha = (n-3)/2$ and $\overline{P}^{(\alpha, \alpha)}_k$ is a
Jacobi polynomial~of degree~$k$, normalized so
that~$\overline{P}_k^{(\alpha,\alpha)}(1) = 1$. Jacobi polynomials are
orthogonal polynomials defined on the interval $[-1,1]$ with respect
to the measure $(1-u^2)^{(n-3)/2} du$.

Now we can extend $A_t$ to $L^2(S^{n-1})$ by setting
\[
(A_t f)(\xi) = \sum_{k = 0}^{\infty} \lambda_k(t) \sum_{l=1}^{c_{k,l}}
(f,S_{k,l}) S_{k,l}(\xi),
\]
where $(f,g) = \int_{S^{n-1}} f(x) \overline{g(x)} \, d\omega(x)$ with
$\omega$ being the rotationally invariant probability measure on the
sphere. Since $\lambda_k(t)$ lies in the interval $[-1,1]$ the
operator $A_t$ is a bounded, self-adjoint operator. One can estimate
the growth of the eigenvalues following Szeg\"o \cite[(4.1.1) and
Theorem 8.21.8]{Szego1938} by
\[
\lambda_k(t) = O(k^{-1/2-(n-3)/2}),
\]
so the operator is even compact when $n \geq 3$.

Now we can define $A_\nu \colon L^2(S^{n-1}) \to L^2(S^{n-1})$ by
\[
(A_{\nu} f)(\xi) = \sum_{k = 0}^{\infty} \int_{-1}^1 \lambda_k(t) \, d\nu(t) \sum_{l=1}^{c_{k,l}}
(f,S_{k,l}) S_{k,l}(\xi),
\]
and it is clear that
\[
m(A_\nu) = \inf_{k \in \N} \int_{-1}^1 \lambda_k(t) \, d\nu(t) \quad \text{and} \quad M(A_\nu) = \sup_{k \in \N} \int_{-1}^1 \lambda_k(t)
\, d\nu(t).
\]

\subsection{The single inner product graph}

Let $D$ be a Borel subset of the interval $[-1, 1]$ where 1 does not
lie in the topological closure of $D$.  Let $\nu$ be a Borel measure
whose support lies in $D$. It is easy to see that the operator $A_\nu$
respects the graph $G(S^{n-1}, D)$. So one can apply~\eqref{eq:alpha-best}
and~\eqref{eq:chi-best} to this graph.

The case when $D$ consists only of the single inner product $t$ is
particularly simple as no optimization over $\nu$ is necessary to find
the optimal spectral bound:
\begin{equation*}
\overline\alpha(G(S^{n-1},\{t\})) \leq \frac{-\inf_{k\in N}
  \lambda_k(t)}{1-\inf_{k \in \N} \lambda_k(t)}
, \;\;
\chim(G(S^{n-1},\{t\})) \geq \frac{1 - \inf_{k \in \N} \lambda_k(t)}{-\inf_{k \in \N} \lambda_k(t)},
\end{equation*}
where $\lambda_k(t) = \overline{P}^{(\alpha,\alpha)}_k(t)$.  This result also
follows from \cite[Theorem~6.2,
Section~10]{BachocNebeOliveiraVallentin}. The two bounds multiply to
one which also follows from Theorem~\ref{th:transitive}. 

The bound is tight in the planar case $n = 2$ for $t = \cos(p/q \pi)$,
with $p,q \in \N$, $\gcd(p,q) = 1$, when $q$ is even, and when $t =
\cos(x \pi)$ when~$x$ is not rational. It is also tight when $n = 3$
and $t = -1/3$. The latter was first proved by Lov\'asz
\cite{Lovasz1983a} using topological methods.

\section*{Acknowledgements}

The authors thank Markus Haase and David Steurer for helpful
discussions. We thank the referee for the kind words, careful reading and
comments which led to improvements of the paper.


\begin{thebibliography}{[99]}

\bibitem{AignerZiegler}
M.~Aigner, G.M.~Ziegler,
{\em Proofs from the book},
Springer, 1998.

\bibitem{AlonMakarychevMakarychevNaor}
N.~Alon, K.~Makarychev, Y.~Makarychev, A.~Naor,
{\em Quadratic forms on graphs},
Invent. Math. {\bf 163} (2006), 499--522. 

\bibitem{AndrewsAskeyRoy}
G.E.~Andrews, R.~Askey, R.~Roy, 
{\em Special functions}, 
Cambridge University Press, 1999.

\bibitem{ArdalManuchRosenfeldShelahStacho}
H.~Ardal, J.~Ma\v{n}uch, M.~Rosenfeld, S.~Shelah, L.~Stacho,
{\em The odd-distance plane graph},
Discr. Comp. Geom. {\bf 42} (2009), 132--141.

\bibitem{BachocPecherThiery}
C.~Bachoc, A.~P\^echer, A.~Thi\'ery,
{\em On the theta number of powers of cycle graphs},
arXiv:1103.0444v1 [math.CO], 17 pages.
(\url{http://arxiv.org/abs/1103.0444})

\bibitem{BachocNebeOliveiraVallentin}
C.~Bachoc, G.~Nebe, F.M.~de Oliveira Filho, F.~Vallentin,
{\it Lower bounds for measurable chromatic numbers},
Geom. Funct. Anal. {\bf 19} (2009), 645--661.
(\url{http://arxiv.org/abs/0801.1059})

\bibitem{Bollobas}
B.~Bollob\'as,
{\em Modern graph theory},
Springer, 1998.

\bibitem{BrietOliveiraVallentin} 
J.~Bri\"et, F.M.~de Oliveira Filho, F.~Vallentin, 
{\em Grothendieck inequalities for semidefinite programs with rank constraint},
arXiv:1011.1754v2 [math.OC], 22 pages.
(\url{http://arxiv.org/abs/1011.1754})

\bibitem{Delsarte}
P.~Delsarte, 
{\em An algebraic approach to the association schemes of coding
  theory},
Philips Res. Rep. Suppl. (1973), vi+97.

\bibitem{EllisFriedgutFilmus}
D.~Ellis, E.~Friedgut, Y.~Filmus,
{\em Triangle intersecting families of graphs},
J. Eur. Math. Soc. (JEMS) {\bf 14} (2012), 841--885.
(\url{http://arxiv.org/abs/1010.4909})

\bibitem{EllisFriedgutPilpel}
D.~Ellis, E.~Friedgut, H.~Pilpel,
{\em Intersecting families of permutations},
J. Amer. Math. Soc. {\bf 24} (2011), 649--682.
(\url{http://arxiv.org/abs/1011.3342})

\bibitem{Folland}
G.B.~Folland,
{\em Real analysis. Modern techniques and their application. Second edition},
John Wiley \& Sons, 1999.

\bibitem{FurstenbergKatznelsonWeiss} H.~Furstenberg, Y.~Katznelson, B.~Weiss, {\em Ergodic
    theory and configurations in sets of positive density}, p.
  184--198 in {\sl Mathematics of Ramsey theory} (J.~Ne\v{s}et\v{r}il, V.~R\"odl
 ed.), Springer, 1989.

\bibitem{Galtman}
A.~Galtman,
{\em Spectral characterizations of the Lov\'asz number and the Delsarte number of a graph}, 
J. Alg. Comb. {\bf 12} (2000), 131--142.

\bibitem{GouveiaParriloThomas}
J.~Gouveia, P.A.~Parrilo, R.R.~Thomas,
{\em Theta bodies for polynomial ideals},
SIAM J. Optim. {\bf 20} (2010), 2097--2118.
(\url{http://arxiv.org/abs/0809.3480})

\bibitem{Groemer1996a}
H.~Groemer,
{\em Geometric Applications of Fourier Series and Spherical
  Harmonics},
Cambridge University Press, 1996.

\bibitem{GroetschelLovaszSchrijver}
M.~Gr\"otschel, L.~Lov\'asz, A.~Schrijver,
{\em Polynomial algorithms for perfect graphs},
Annals Discrete Mathematics {\bf 21} (1984) 325--356.

\bibitem{Halmos}
P.R.~Halmos,
{\em A Hilbert space problem book},
Springer, 1982

\bibitem{Hoffman}
A.J.~Hoffman,
{\em On eigenvalues and colorings of graphs},
pp.\ 79--91 in (B. Harris, ed.) Graph Theory and its Applications, Academic Press, 1970.

\bibitem{KargerMotwaniSudan}
D.~Karger, R.~Motwani, M.~Sudan,
{\em Approximate graph coloring by semidefinite programming},
Journal of the ACM {\bf 45} (1998), 246--265.
(\url{http://arxiv.org/abs/cs/9812008})

\bibitem{KleinbergGoemans}
J.~Kleinberg, M.X.~Goemans. 
{\em The Lov\'asz theta function and a semidefinite programming relaxation of vertex cover},
SIAM J. Discrete Math.~{\bf 11} (1998), 196--204.

\bibitem{Knuth} 
D.E.~Knuth, 
{\em The sandwich theorem},
Electron.~J.~Combin.~\textbf{1} (1994), 48 pp.

\bibitem{Kolountzakis}
M.N. Kolountzakis,
{\em Distance sets corresponding to convex bodies},
Geom. Funct. Anal. {\bf 14} (2004), 734--744. 
(\url{http://arxiv.org/abs/math/0303212})

\bibitem{Lovasz} 
L.~Lov\'asz,
{\em On the Shannon capacity of a graph},
IEEE Trans.~Inf.~Th.~{\bf 25} (1979), 1--7.

\bibitem{Lovasz1983a}
L.~Lov\'asz,
{\em Self-dual polytopes and the chromatic number of distance graphs on the sphere},
Acta Sci. Math. {\bf 45} (1983), 317--323.

\bibitem{McElieceRodemichRumsey}
R.J.~McEliece, E.R.~Rodemich, H.~C. Rumsey Jr., 
{\em The Lov\'asz bound and some generalizations},
J. Combin. Inf. System Sci.~{\bf 3} (1978) 134--152.

\bibitem{OliveiraVallentin}
F.M. de Oliveira Filho, F. Vallentin,
{\em Fourier analysis, linear programming, and densities of distance avoiding sets in $\R^n$},
J. Eur. Math. Soc. (JEMS) {\bf 12} (2010), 1417--1428.
(\url{http://arxiv.org/abs/0808.1822})

\bibitem{OliveiraVallentin2013a}
F.M. de Oliveira Filho, F. Vallentin,
{\em A quantitative version of Steinhaus' theorem for compact,
  connected, rank-one symmetric spaces},
to appear in Geom. Dedicata.
(\url{http://arxiv.org/abs/1005.0471})

\bibitem{Schrijver1979a} 
A.~Schrijver,
{\em A comparison of the Delsarte and Lov\'asz bounds},
IEEE Trans.~Inf.~Th.~{\bf 25} (1979), 425--429.

\bibitem{Schrijver2003}
A.~Schrijver, 
{\em Combinatorial Optimization: Polyhedra and Efficiency},
Springer-Verlag, Berlin, 2003.

\bibitem{Soifer}
A. Soifer, 
{\em The Mathematical Coloring Book}, 
Springer, 2009.

\bibitem{Steinhardt}
J.~Steinhardt,
{\em On coloring the odd-distance graph},
Electron. J. of Combin. {\bf 16} (2009) N12, 7 pp.
(\url{http://arxiv.org/abs/0908.1452})

\bibitem{Szego1938}
G.~Szeg\"o,
{\em Orthogonal Polynomials},
American Mathematical Society Colloquium Publications Volume XXIII,
American Mathematical Society, Providence, 1975.

\bibitem{Szekely} L.A.~Sz\'ekely, {\em Erd\H{o}s on unit distances and
the Szemer\'edi-Trotter theorems}, p. 649--666 in {\sl Paul
Erd\H{o}s and his mathematics} (G.~Hal\'asz, L.~Lov\'asz, M.~Simonovits, V.T.~S\'os ed.), Springer, 2002.

\end{thebibliography}
\end{document}